\documentclass{interact}
\usepackage{bm}
\def\<{\langle}
\def\>{\rangle}
\usepackage{upgreek}

\usepackage{amsmath, amsthm, amscd, amsfonts, amssymb, graphicx, color}
\usepackage[bookmarksnumbered, colorlinks, plainpages]{hyperref}
\hypersetup{colorlinks=true,linkcolor=red, anchorcolor=green, citecolor=red, urlcolor=red, filecolor=magenta, pdftoolbar=true}

\def\1{\mathbf{1}}

\theoremstyle{plain}
\newtheorem{thm}{Theorem}[section]

\newtheorem{lem}[thm]{Lemma}

\newtheorem{rmk}[thm]{Remark}
\theoremstyle{definition}

\theoremstyle{remark}

\numberwithin{equation}{section}

\begin{document}
	\title{ Maximal Algebras of Block Toeplitz Matrices with Entries in the Schur Algebra}
	
	\author{
		\name{Muhammad Ahsan Khan}
		\affil{Department of Mathematics, University of Kotli Azad Jammu $\&$ Kashmir, Kotli 11100, Azad Jammu $\&$ Kashmir, Pakistan, Email ahsankhan388@hotmail.com}
	}
	
	\maketitle
	
	\begin{abstract}
	The classification of maximal algebras of square block Toeplitz matrices is a considerably more difficult problem and has received relatively little attention in the existing literature. In this work, we approach the problem under the assumption that the entries belong to the Schur algebra. Within these settings, we obtain a complete classification of all maximal algebras of such block Toeplitz matrices.
	\end{abstract}
	
	
	\begin{amscode}
		15A27, 15A30, 15B05, 16N40
	\end{amscode}	
	\section{Introduction}	
	A square matrix is called Toeplitz if its entries along the negative sloping diagonals have same value, that is, 
	\[
	\begin{pmatrix}
		t_0& t_{-1}& t_{-2} & \cdots &t_{1-n}\\
		t_1& t_0 & t_{-1} & \cdots &t_{2-n}\\
		t_{2}& t_{1} & t_{0} & \cdots &t_{3-n}\\
		\vdots & \vdots & \vdots &\ddots & \vdots\\
		t_{n-1} & t_{n-2}& t_{n-3} & \cdots &t_{0}
	\end{pmatrix}
	\]
These matrices are one of the most studied classes of structured matrices, with applications in various areas of mathematics and engineering \cite{BottcherSilbermann, GrenanderSzego, Iohvidov}. The multiplicative properties of Toeplitz matrices have been investigated in \cite{Shalom, GuPatton}. In particular, it is known that the product of two Toeplitz matrices is not necessarily Toeplitz; this leads to the natural question of describing the maximal commutative subalgebras consisting entirely of Toeplitz matrices. For scalar Toeplitz matrices, such algebras are completely classified and are known as \emph{generalized circulants} \cite{Shalom}.

 The corresponding general theory of block Toeplitz matrices remains comparatively less explored in the existing literature (see \cite{BottcherSilbermann,GohbergGoldbergKaashoek}). One of the principal reasons for this is the additional structural difficulties that arise in the block setting, which are not present in the scalar case. The presence of non scalar entries significantly complicates the algebraic structure of such matrices, thereby making their analysis substantially more challenging. Despite these difficulties, the study of block Toeplitz matrices is of considerable theoretical interest. Moreover, beyond its intrinsic mathematical significance, this area also plays an important role in applications, particularly in multivariate control theory and related areas of systems analysis \cite{Wen, Wex, Khan2}. The multiplicative properties of these matrices have been discussed in \cite{Khan, KhanTimotin, Khan1, Khanameur, Khan2} The problem of determining all maximal algebras generated by such matrices was posed as an open question in \cite{Shalom}. However, the problem appears to be rather challenging, and only limited progress has been achieved since then. A partial development in this direction is presented in \cite{Khan}, where a particular class of maximal algebras is characterized. In the present work, we continue this line of investigation by deriving certain general structural properties of these algebras and by obtaining classifications in specific cases.

 In the present paper we treat a natural class: we assume that the entries belong to a fixed maximal commutative subalgebra of matrices, namely a \emph{Schur algebra} \(\mathcal{O}_{\sigma,\tau}\). These algebras, introduced in \cite{BrownCall}, are maximal commutative subalgebras of complex matrices of size $d^2$ that have a particularly simple block structure. Using the general settings developed in \cite{Khan}, we prove that every maximal commutative subalgebra of block Toeplitz matrices with entries in \(\mathcal{O}_{\sigma,\tau}\) is either of the form \(\mathcal{F}_{A,B}\) (with at least one of \(A,B\) invertible) or the special algebra \(\mathcal{B}\otimes \mathcal{O}_{\sigma, \tau}\) consisting of matrices whose off diagonal blocks are nilpotent. The two families are disjoint and exhaust all possibilities.

The paper is organized as follows. Section 2 contains the necessary preliminaries on scalar Toeplitz matrices, block Toeplitz matrices, and Schur algebras. In Section 3 we recall the construction \(\mathcal{F}_{A,B}\) introduced in \cite{Khan} and prove its basic properties. Section 4 is devoted to the special algebra \(\mathcal{B}\otimes \mathcal{O}_{\sigma, \tau}\) and its description as a \(\mathcal{F}_{A,B}\) algebra for nilpotent generators. Finally, in Section 5 we  prove the main classification theorem.
	\section{Notations and Preliminaries}	

		Let \(\mathbb{C}\) denote the field of complex numbers. For positive integers $\sigma, \tau$, let $M_{\sigma\times \tau}\otimes\mathbb{C}$ denote the space of all $\sigma\times\tau$ complex matrices. If  $\sigma=\tau=d$, then \(M_d\otimes \mathbb{C}\) denotes the algebra of all \(d\times d\) complex matrices. 	A scalar Toeplitz matrix of size \(n\) is of the form
		\[
		T = (t_{p-q})_{p,q=0}^{n-1}, \qquad t_j\in\mathbb{C}.
		\]
		The linear space of all such matrices is denoted \(\mathcal{T}_n\otimes \mathbb{C}\). For a subalgebra \(\mathcal{A}\), its commutant is \(\mathcal{A}' = \{X\in M_d\otimes \mathbb{C} : XB=BX \text{ for all } B\in\mathcal{A}\}\). It is not difficult to show that $\mathcal{A}'$ is also an algebra. In this work, we shall primarily be concerned with block matrices, that is, matrices whose entries are not necessarily scalars but elements of a given algebra, i.e., a block Toeplitz matrix is an \(nd\times nd\) matrix partitioned into \(n\times n\) blocks, each of size \(d\times d\), constant along diagonals:
		\[
		\mathbf{T} = (T_{p-q})_{p,q=0}^{n-1}, \qquad T_j\in M_d \otimes \mathbb{C}.
		\]  
		If $\mathcal{A}$  is a complex algebra, we shall employ the following notation:
		\begin{itemize}
			\item $\mathcal{M}_n\otimes \mathcal{A}$ is the collection of $n\times n$ block matrices whose entries all belong to $\mathcal{A}$ ;
			\item  $\mathcal{T}_n\otimes \mathcal{A}$ is the collection of $ n\times n$ block Toeplitz matrices whose entries all belong to $\mathcal{A}$.
		\end{itemize}
A block matrix $\mathbf{T} = (T_{i,j})_{i,j=0}^{n-1}$
has $2n-1$ diagonals, corresponding to $i-j=k\quad (k=1-n, \cdots, 0, \cdots, n-1)$. A block Toeplitz matrix has all its diagonals constant.
	
	The product of two block Toeplitz matrices is not necessarily block Toeplitz. A necessary and sufficient condition for the product to be block Toeplitz is given by the following Lemma from \cite{Khan, Shalom}.
	
	\begin{lem}
		\label{lem:product}
		Let \(\mathbf{T}=(T_{p-q})\) and \(\mathbf{U}=(U_{p-q})\) be in \(\mathcal{T}_{n}\otimes \mathcal{A} \). Then \(\mathbf{T}\mathbf{U}\) is block Toeplitz if and only if
		\[
		T_p U_{q-n} = T_{p-n} U_q \qquad \text{for all } p,q=1,\dots,n-1.
		\]
	\end{lem}
	
As mentioned in the introduction, the problem of classifying all maximal algebras of block Toeplitz matrices is considerably difficult to handle. In the present work, we address this problem under the assumption that the entries of the block Toeplitz matrices belong to the Schur algebra, which is defined as follows: 

	Fix positive integers \(\sigma,\tau\) with \(\sigma+\tau=d\) and \(|\sigma-\tau|\le 1\). The \emph{Schur algebra} \(\mathcal{O}_{\sigma,\tau}\) is defined as the set of all matrices of the form
	\[
	\begin{pmatrix} \lambda I_\sigma & X \\ 0 & \lambda I_\tau \end{pmatrix},
	\qquad \lambda\in\mathbb{C},\; X\in M_{\sigma\times\tau} \otimes \mathbb{C}.
	\]
	It is known \cite{BrownCall} that \(\mathcal{O}_{\sigma,\tau}\) is a maximal commutative subalgebra of \(M_d \otimes \mathbb{C}\). Its radical \(\mathcal{R}\) consists of those matrices with \(\lambda=0\); note that \(\mathcal{R}^2=0\) and every element of \(\mathcal{R}\) is nilpotent. The following simple criterion from \cite{Khan} is essential.
	
	\begin{lem}\label{lem:invert}
		An element \(T\in\mathcal{O}_{\sigma,\tau}\) is invertible if and only if its scalar part \(\lambda\neq 0\). Consequently, if \(T\) is not invertible, then its kernel contains the subspace \(\mathbb{C}^\sigma\oplus\{0\}\).
	\end{lem}
From here onward we fix a maximal commutative subalgebra \(\mathcal{A}\subseteq M_d \otimes \mathbb{C}\). For any \(A,B\in\mathcal{A}'\) we define
	\[
	\mathcal{F}_{A,B}^{\mathcal{A}} = \bigl\{ \mathbf{T}=(T_{p-q})\in\mathcal{T}_{n} \otimes \mathcal{A} : A T_j = B T_{j-n} \text{ for } j=1,\dots,n-1 \bigr\}.
	\]
	When \(\mathcal{A}=\mathcal{O}_{\sigma,\tau}\), we simply write \(\mathcal{F}_{A,B}\). The following Lemma from \cite{Khan} is quite elementary.
	
	\begin{lem} \label{lem:kern}
		If \(A,B\in M_d \otimes \mathbb{C}\) satisfy \(Ker A\cap Ker B=\{0\}\), then for any \(d\times d\) matrix \(T\) with \(AT=BT=0\) we have \(T=0\).
	\end{lem}
	
	The next Theorem from \cite{Khan} shows that under the kernel condition given in Lemma \ref{lem:kern}, \(\mathcal{F}_{A,B}^{\mathcal{A}}\) is an algebra.
	
	\begin{thm}\label{thm:alg}
		Let \(\mathcal{A}\) be a commutative subalgebra of \(M_d\otimes \mathbb{C}\) and let \(A,B\in\mathcal{A}'\). If \( Ker A\cap Ker B=\{0\}\), then \(\mathcal{F}_{A,B}^{\mathcal{A}}\) is closed under multiplication, hence a commutative algebra.
	\end{thm}
	
	Maximality of such algebras is characterized by the following result from \cite{Khan}.
	
	\begin{thm}
		\label{thm:max}
		Let \(\mathcal{A}\) be a maximal commutative subalgebra of \(M_d \otimes \mathbb{C}\) and let \(A,B\in\mathcal{A}'\). Then \(\mathcal{F}_{A,B}^{\mathcal{A}}\) is a maximal commutative subalgebra of \(\mathcal{T}_{n}\otimes\mathbb{C} \) if and only if \(Ker A\cap Ker B=\{0\}\).
	\end{thm}
	
The below result from \cite{Khan} shows that when a commutative subalgebra of block Toeplitz matrices contains an element with an invertible off diagonal block, it must be contained in some \(\mathcal{F}_{A,B}^{\mathcal{A}}\).
	
	\begin{thm}
		\label{thm:contain}
		Let \(\mathcal{A}\) be a maximal commutative subalgebra of \(M_d \otimes \mathbb{C}\) and let \(\mathcal{C}\subseteq\mathcal{T}_{n}\otimes \mathcal{A}\) be a commutative subalgebra. If \(\mathcal{C}\) contains an element \(\mathbf{T}\) such that for some \(j\neq 0\) the block \(T_j\) is invertible in \(\mathcal{A}\), then there exist \(A,B\in\mathcal{A}\) satisfying \(Ker A\cap Ker B=\{0\}\) such that \(\mathcal{C}\subseteq\mathcal{F}_{A,B}^{\mathcal{A}}\).
	\end{thm}
	
		\section{The special algebra \(\mathcal{B}\otimes \mathcal{O}_{\sigma, \tau}\)}
	Now we specialize to the Schur algebra \(\mathcal{A}=\mathcal{O}_{\sigma,\tau}\). Recall that its radical \(\mathcal{R}\) consists of strictly upper block triangular matrices with zero scalar part. Define
	\[
	\mathcal{B}\otimes \mathcal{O}_{\sigma, \tau} = \bigl\{ \mathbf{T}=(T_{p-q})\in\mathcal{T}_{n} \otimes \mathcal{O}_{\sigma, \tau} : T_j\in\mathcal{R} \text{ for all } j\neq 0 \bigr\}.
	\]
	Thus $\mathcal{B}\otimes \mathcal{O}_{\sigma, \tau}$ consists of block Toeplitz matrices whose off diagonal blocks are nilpotent (i.e., have \(\lambda=0\)); the main diagonal blocks are arbitrary elements of \(\mathcal{O}_{\sigma, \tau}\). 
	
	The next Theorem shows that $\mathcal{B}\otimes \mathcal{O}_{\sigma, \tau}$ is a maximal commutative algebra and can be realized as \(\mathcal{F}_{A,B}\) for any linearly independent nilpotent matrices.
	
	\begin{thm} \label{thm:S}
		The set $\mathcal{B}\otimes \mathcal{O}_{\sigma, \tau}$ is a maximal commutative subalgebra of \(\mathcal{T}_{n}\otimes \mathcal{O}_{\sigma, \tau}\). Moreover, for any two linearly independent matrices \(A,B\in\mathcal{R}\), we have \(\mathcal{B}\otimes \mathcal{O}_{\sigma, \tau}=\mathcal{F}_{A,B}\). In particular, \(\mathcal{F}_{A,B}\) is an algebra even though \( Ker A\cap Ker B\neq\{0\}\).
	\end{thm}
	
	\begin{proof}
		We first show that \(\mathcal{B}\otimes \mathcal{O}_{\sigma, \tau}\) is an algebra. Obviously  \(\mathcal{B}\otimes \mathcal{O}_{\sigma, \tau}\) is a linear subspace inside $\mathcal{T}_n\otimes \mathcal{O}_{\sigma, \tau}$. For closeness,  take any $\mathbf{T}=(T_{p-q})\), \(\mathbf{U}=(U_{p-q})$ in $\mathcal{B}\otimes \mathcal{O}_{\sigma, \tau}$. For any \(p,q=1,\cdots,n-1\), we have \(T_p, U_{q-n}, T_{p-n}, U_q\in\mathcal{R}\). Since \(\mathcal{R}^2=0\), it follows that
		\[
		T_p U_{q-n}=0,\qquad T_{p-n} U_q=0.
		\]
		Thus condition (3.1) of Lemma \ref{lem:product} holds, so \(\mathbf{T}\mathbf{U}\) is block Toeplitz. Moreover, for any off diagonal position \((i,j)\) with \(i\neq j\), the block \((\mathbf{T}\mathbf{U})_{i,j}\) is a sum of terms of the form \(T_{i-\ell} U_{\ell-j}\). In each such term, at least one of the indices \(i-\ell\) or \(\ell-j\) is nonzero, so the corresponding factor lies in \(\mathcal{R}\). Because \(\mathcal{R}\) is an ideal of \(\mathcal{O}_{\sigma, \tau}\), each term belongs to \(\mathcal{R}\). Hence \((\mathbf{T}\mathbf{U})_{i,j}\in\mathcal{R}\) for \(i\neq j\), and therefore $\mathbf{T}\in\mathcal{B}\otimes \mathcal{O}_{\sigma, \tau}$. This proves closure under multiplication.
		
		We now establish the maximality of $\mathcal{B}\otimes \mathcal{O}_{\sigma, \tau}$. To do this end, let $\mathbf{U}=(U_{p-q})\in\mathcal{T}_{n}\otimes \mathcal{O}_{\sigma, \tau}$ be such that \(\mathbf{U}\) commutes with every element of $\mathcal{B}\otimes \mathcal{O}_{\sigma, \tau}$. Writing each block in the canonical form
		\[
		U_r = \begin{pmatrix} \lambda_r I_\sigma & X_r \\ 0 & \lambda_r I_\tau \end{pmatrix}, \qquad \lambda_r\in\mathbb{C},\; X_r\in M_{\sigma\times\tau}\otimes \mathbb{C}.
		\]
		Fix an index $r$ with $1\le r\le n-1$. Choose a nonzero nilpotent matrix \(N\in\mathcal{R}\) of the simple form
		\[
		N = \begin{pmatrix} 0 & I_{\sigma\times\tau} \\ 0 & 0 \end{pmatrix},
		\]
		where \(I_{\sigma\times\tau}\) denotes any fixed nonzero matrix (for instance, the matrix with a single 1 in the \((1,1)\) position if \(\sigma,\tau\ge1\)). Define a matrix $\mathbf{T}\in\mathcal{B}\otimes \mathcal{O}_{\sigma, \tau}$ by setting \(T_{1-n}=N\) and \(T_j=0\) for all \(j\neq 1-n\)). Since \(\mathbf{U}\) commutes with \(\mathbf{T}\), we have \(\mathbf{T}\mathbf{U}=\mathbf{U}\mathbf{T}\). Because \(\mathbf{T}\) has only one nonzero diagonal, the product \(\mathbf{T}\mathbf{U}\) is block Toeplitz; consequently, by Lemma \ref{lem:product} applied to the pair \((\mathbf{T},\mathbf{U})\), condition (3.1) must hold. In particular, for \(p=1\) and \(q=r\) we obtain
		\[
		T_1 U_{r-n} = T_{1-n} U_r.
		\]
		But \(T_1=0\) and \(T_{1-n}=N\), so this becomes
		\[
		0 = N U_r.
		\]
		A direct computation gives
		\[
		N U_r = \begin{pmatrix} 0 & I \\ 0 & 0 \end{pmatrix} \begin{pmatrix} \lambda_r I & X_r \\ 0 & \lambda_r I \end{pmatrix} = \begin{pmatrix} 0 & \lambda_r I \\ 0 & 0 \end{pmatrix} = \lambda_r N.
		\]
		Since \(N\neq0\), we deduce \(\lambda_r=0\). A similar argument, using a matrix \(\mathbf{T}\) with \(T_1=N\) and all other blocks zero, shows that \(\lambda_{-r}=0\) for \(r\ge1\). Thus all off diagonal blocks of \(\mathbf{U}\) have zero scalar part, that is,  lie in \(\mathcal{R}\). Hence $\mathbf{U}\in\mathcal{B}\otimes \mathcal{O}_{\sigma, \tau}$. This proves that $\mathcal{B}\otimes \mathcal{O}_{\sigma, \tau}$ is maximal commutative.
		
		 We now show that for linearly independent matrices \(A,B\in\mathcal{R}\) we have  $\mathcal{B}\otimes \mathcal{O}_{\sigma, \tau}=\mathcal{F}_{A,B}$.  Let \(A,B\in\mathcal{R}\) be linearly independent. Write them as
		\[
		A = \begin{pmatrix} 0 & A' \\ 0 & 0 \end{pmatrix},\qquad B = \begin{pmatrix} 0 & B' \\ 0 & 0 \end{pmatrix}
		\]
		with \(A',B'\in M_{\sigma\times\tau}\otimes \mathbb{C}\) linearly independent. For any \(\mathbf{T}\in\mathcal{B}\otimes \mathcal{O}_{\sigma, \tau}\), since $\mathcal{R}^2=0$, we have \(T_j\in\mathcal{R}\) for all \(j\neq0\), so \(A T_j = 0\) and \(B T_{j-n}=0\). Hence \(\mathbf{T}\in\mathcal{F}_{A,B}\). Thus \(\mathcal{B}\otimes \mathcal{O}_{\sigma, \tau}\subseteq\mathcal{F}_{A,B}\).
		
		Conversely, let \(\mathbf{T}\in\mathcal{F}_{A,B}\). Write \(T_j = \begin{pmatrix} \lambda_j I\sigma & X_j \\ 0 & \lambda_j I_\tau \end{pmatrix}\). The defining relation \(A T_j = B T_{j-n}\) for \(j=1,\cdots,n-1\) yields
		\[
		\begin{pmatrix} 0 & A' \\ 0 & 0 \end{pmatrix} \begin{pmatrix} \lambda_j I & X_j \\ 0 & \lambda_j I \end{pmatrix} = \begin{pmatrix} 0 & B' \\ 0 & 0 \end{pmatrix} \begin{pmatrix} \lambda_{j-n} I & X_{j-n} \\ 0 & \lambda_{j-n} I \end{pmatrix}.
		\]
		Computing the left hand hand side gives \(\begin{pmatrix} 0 & \lambda_j A' \\ 0 & 0 \end{pmatrix}\); the right-hand side gives \(\begin{pmatrix} 0 & \lambda_{j-n} B' \\ 0 & 0 \end{pmatrix}\). Hence \(\lambda_j A' = \lambda_{j-n} B'\). Because \(A'\) and \(B'\) are linearly independent, we must have \(\lambda_j = \lambda_{j-n} = 0\) for every \(j=1,\cdots,n-1\). Therefore all off diagonal blocks of \(\mathbf{T}\) lie in \(\mathcal{R}\), that is, \(\mathbf{T}\in\mathcal{B}\otimes \mathcal{O}_{\sigma, \tau}\). So \(\mathcal{F}_{A,B}\subseteq \mathcal{B}\otimes \mathcal{O}_{\sigma, \tau}\). 
	\end{proof}
	
	\section{Classification theorem}
	This section is the most important section of this paper as it deals with classification of all the maximal commutative algebras inside $\mathcal{T}_n\otimes\mathcal{O}_{\sigma, \tau}$. 
	
	\begin{thm}\label{thm:main}
		Every maximal commutative subalgebra \(\mathcal{M}\) of \(\mathcal{T}_{n}\otimes \mathcal{O}_{\sigma, \tau}\) is either
		\begin{itemize}
			\item[(i)] \(\mathcal{F}_{A,B}\) for some \(A,B\in\mathcal{O}_{\sigma, \tau}\) with at least one of them invertible, or
			\item[(ii)] the algebra $\mathcal{B}\otimes \mathcal{O}_{\sigma, \tau}$ (which coincides with \(\mathcal{F}_{A,B}\) for any pair of linearly independent nilpotent matrices \(A,B\in\mathcal{R}\)).
		\end{itemize}
		Conversely, every algebra of these two types is maximal commutative.
	\end{thm}
	
	\begin{proof}
		We first prove that each type is maximal
		
 If \(A,B\in\mathcal{O}_{\sigma, \tau}\) satisfy \( Ker A\cap Ker B=\{0\}\), then by Lemma~\ref{lem:invert} this condition is equivalent to at least one of \(A,B\) being invertible. Applying Theorem~\ref{thm:max} with \(\mathcal{A}=\mathcal{O}_{\sigma, \tau}\) we conclude that \(\mathcal{F}_{A,B}\) is a maximal commutative subalgebra of \(\mathcal{T}_{n}\otimes \mathcal{O}_{\sigma, \tau}\).
 
For any two linearly independent \(A,B\in\mathcal{R}\), Theorem~\ref{thm:S} tells us that \(\mathcal{B}\otimes\mathcal{O}_{\sigma, \tau}=\mathcal{F}_{A,B}\) and that \(\mathcal{B}\otimes\mathcal{O}_{\sigma, \tau}\) is maximal commutative.
Thus every algebra of type (i) or (ii) is indeed maximal.
		
		Now  let us assume that $\mathcal{M}\subseteq\mathcal{T}_{n}\otimes \mathcal{O}_{\sigma, \tau}$ be a maximal commutative subalgebra. We distinguish two cases.
		
		\medskip
		\noindent\text{Case 1.} \(\mathcal{M}\) contains an element \(\mathbf{T}\) such that for some index \(j\neq 0\) the block \(T_j\) is invertible in \(\mathcal{O}_{\sigma, \tau}\). By replacing \(\mathbf{T}\) by a suitable power we may assume without loss of generality that \(T_1\) is invertible. Applying Theorem~\ref{thm:contain} with \(\mathcal{A}=\mathcal{O}_{\sigma, \tau}\), we obtain matrices \(A,B\in\mathcal{O}_{\sigma, \tau}\) with \( Ker A\cap Ker B=\{0\}\) such that \(\mathcal{M}\subseteq\mathcal{F}_{A,B}\). Since \(\mathcal{F}_{A,B}\) is itself a commutative algebra (Theorem~\ref{thm:alg}) and \(\mathcal{M}\) is maximal, the inclusion forces \(\mathcal{M}=\mathcal{F}_{A,B}\). By Lemma~\ref{lem:invert}, the condition \( Ker A\cap Ker B=\{0\}\) is equivalent to at least one of \(A,B\) is invertible. Hence \(\mathcal{M}\) is of type (i).
		
		\medskip
		\noindent\text{Case 2.} No element of \(\mathcal{M}\) has an invertible off diagonal block. That is, for every \(\mathbf{T}\in\mathcal{M}\) and every \(j\neq 0\), the block \(T_j\) is not invertible. By Lemma~\ref{lem:invert}, this means that the scalar part of \(T_j\) is zero, i.e. \(T_j\in\mathcal{R}\). Consequently \(\mathcal{M}\subseteq\mathcal{B}\otimes\mathcal{O}_{\sigma, \tau}\). But \(\mathcal{B}\otimes\mathcal{O}_{\sigma, \tau}\) is a maximal commutative algebra (Theorem~\ref{thm:S}), so the inclusion \(\mathcal{M}\subseteq \mathcal{B}\otimes\mathcal{O}_{\sigma, \tau}\) together with maximality of \(\mathcal{M}\) yields \(\mathcal{M}=\mathcal{B}\otimes\mathcal{O}_{\sigma, \tau}\). Thus \(\mathcal{M}\) is of type (ii).
		
		 These two families are disjoint: any \(\mathcal{F}_{A,B}\) with an invertible off‑diagonal block (e.g. the generator itself) cannot lie in \(\mathcal{B}\otimes\mathcal{O}_{\sigma, \tau}\) because all off‑diagonal entries of \(\mathcal{B}\otimes\mathcal{O}_{\sigma, \tau}\) are nilpotent. Conversely, \(\mathcal{B}\otimes\mathcal{O}_{\sigma, \tau}\) contains no element with an invertible off‑diagonal block, so it cannot be of the first type. 
	\end{proof}
	
		\begin{rmk}
		The parameterization in case (ii) is not unique: any linearly independent pair $A,B\in\mathcal{R}$ yields the same algebra \(\mathcal{B}\otimes\mathcal{O}_{\sigma, \tau}\). In contrast, for case (i) different pairs \((A,B)\) may give distinct algebras; however, they are all maximal. Note that two pairs $(A, B)$ and $(A^\prime, B^\prime)$ satisfying the nondegeneracy condition $(Ker A\cap Ker B=\{0\}$ and $Ker A^\prime\cap Ker B^\prime=\{0\})$ give the same algebra $\mathcal{F}_{A,B}=\mathcal{F}_{A^\prime, B^\prime}$ if and only if $AB^\prime=A^\prime B$. This follows from Theorem~5.1 (ii) of \cite{KhanTimotin}; a short proof in our case is obtained by observing that the element $\mathbf{T}$ with $T_1=B$ and $T_{1-n}=A$ is in $\mathcal{F}_{A,B}=\mathcal{F}_{A^\prime, B^\prime}$ yields $AB^\prime=A^\prime B$. The parametrization is therefore unique  upto the equivalence relation $(A, B)\sim (A^\prime, B^\prime)$ defined by $AB^\prime=A^\prime B$.
	\end{rmk}
	We support our main result by constructing explicit examples of maximal commutative subalgebras of block Toeplitz matrices with entries in a Schur algebra.  
	Take \(\sigma = 2\), \(\tau = 1\), so \(d = 3\). The Schur algebra \(\mathcal{O}_{2,1}\) consists of all \(3 \times 3\) matrices of the form  
	\[
	\begin{pmatrix}
		\lambda I_2 & X \\
		0 & \lambda
	\end{pmatrix},
	\quad \lambda \in \mathbb{C},\; X \in M_{2\times 1}\otimes \mathbb{C}.
	\]  
	Its radical \(\mathcal{R}\) is the set of matrices with \(\lambda = 0\), i.e.,  
	\[
	\mathcal{R} = \left\{ \begin{pmatrix} 0 & x_1 \\ 0 & x_2 \\ 0 & 0 \end{pmatrix} : x_1, x_2 \in \mathbb{C} \right\}
	\]  
	(where the column \(X = (x_1, x_2)^T\)).  
	
	Let \(n = 3\). Then \(\mathcal{T}_{3}\otimes \mathcal{O}_{2,1}\) consists of \(3 \times 3\) block Toeplitz matrices whose blocks are in \(\mathcal{O}_{2,1}\).
	
	\subsection*{Example 1: Type (i) – one generator invertible}
	Choose \(A = I_3\) (invertible) and \(B = 0\). Then \(\mathcal{F}_{I,0}\) is defined by  
	\[
	I \cdot T_j = 0 \cdot T_{j-3} \quad\Longrightarrow\quad T_j = 0 \text{ for } j = 1,2.
	\]  
	Thus every matrix in \(\mathcal{F}_{I,0}\) has the form  
	\[
	\mathbf{T} = \begin{pmatrix}
		T_0 & 0 & 0 \\
		0 & T_0 & 0 \\
		0 & 0 & T_0
	\end{pmatrix},
	\]  
	where \(T_0 \in \mathcal{O}_{2,1}\) is arbitrary. This algebra is isomorphic to \(\mathcal{O}_{2,1}\) and is maximal commutative.
	
	\subsection*{Example 2: Type (i) – both invertible}
	Take \(A = I_3\), \(B = \mu I_3\) with \(\mu \neq 0\). Then the defining relation is  
	\[
	T_j = \mu T_{j-3} \quad (j=1,2).
	\]  
	All blocks are determined by \(T_0, T_1, T_2\) (with \(T_{-1} = \mu T_2\), \(T_{-2} = \mu T_1\)). Thus  
	\[
	\mathbf{T} = \begin{pmatrix}
		T_0 & T_1 & T_2 \\
		\mu T_2 & T_0 & T_1 \\
		\mu T_1 & \mu T_2 & T_0
	\end{pmatrix},
	\]  
	with \(T_0, T_1, T_2 \in \mathcal{O}_{2,1}\). This is a 3‑parameter family, maximal commutative.
	
	\subsection*{Example 3: Type (ii) – the special algebra \(\mathcal{B}\otimes\mathcal{O}_{\sigma, \tau}\)}
	\(\mathcal{B}\otimes\mathcal{O}_{2,1}\) consists of all block Toeplitz matrices whose off diagonal blocks lie in \(\mathcal{R}\). A generic element is  
	\[
	\mathbf{T} = \begin{pmatrix}
		D_0 & N_1 & N_2 \\
		N_{-1} & D_0 & N_1 \\
		N_{-2} & N_{-1} & D_0
	\end{pmatrix},
	\]  
	where \(D_0 \in \mathcal{O}_{2,1}\) and each \(N_j \in \mathcal{R}\). For instance, take  
	\[
	D_0 = \begin{pmatrix} \lambda & 0 & 0 \\ 0 & \lambda & 0 \\ 0 & 0 & \lambda \end{pmatrix},\quad
	N_1 = \begin{pmatrix} 0 & 0 & a \\ 0 & 0 & b \\ 0 & 0 & 0 \end{pmatrix},\quad
	N_2 = \begin{pmatrix} 0 & 0 & c \\ 0 & 0 & d \\ 0 & 0 & 0 \end{pmatrix},
	\]  
	and similarly for \(N_{-1}, N_{-2}\) (they may be chosen independently). The product of any two such matrices remains in \(\mathcal{B}\otimes\mathcal{O}_{2,1}\) because any product of nilpotents is zero, and a nilpotent times a diagonal block stays nilpotent.  
	
	Moreover, \(\mathcal{B}\otimes\mathcal{O}_{2,1}\) can also be written as \(\mathcal{F}_{A,B}\) for any two linearly independent nilpotent matrices. Choose  
	\[
	A = \begin{pmatrix} 0 & 1 & 0 \\ 0 & 0 & 0 \\ 0 & 0 & 0 \end{pmatrix},\quad
	B = \begin{pmatrix} 0 & 0 & 1 \\ 0 & 0 & 0 \\ 0 & 0 & 0 \end{pmatrix}.
	\]  
	Then the condition \(A T_j = B T_{j-3}\) forces the scalar part of \(T_j\) to be zero, so exactly the matrices in \(\mathcal{B}\otimes\mathcal{O}_{\sigma, \tau}\) satisfy it.
	\begin{rmk}
	We have obtained a complete classification of maximal commutative subalgebras of block Toeplitz matrices with entries in a Schur algebra. This adds to the list of cases where such a classification is possible. In \cite{KhanTimotin} we treated the case where the entries belong to a singly generated algebras. The general problem for an arbitrary maximal commutative subalgebra of \(M_d\otimes \mathbb{C}\) remains open and appears to be considerably more difficult. 
\end{rmk}
	
	\section*{Acknowledgments}
	The author is highly grateful to referee for his valuable
	suggestions. 

\end{document}